\documentclass{amsart}

\usepackage[all]{xy}

\usepackage{hyperref}

\usepackage{cleveref}

\usepackage{amssymb}

\theoremstyle{plain}
\newtheorem{theorem}[subsection]{Theorem}

\newtheorem{lemma}[subsection]{Lemma}
\newtheorem{corollary}[subsection]{Corollary}
\theoremstyle{definition}

\newcommand{\Aut}{\mathrm{Aut}}
\newcommand{\Bun}{\mathrm{Bun}}
\newcommand{\et}{{\mathrm{\acute{e}t}}}

\newcommand{\Gm}{{\mathbb{G}_m}}
\newcommand{\Hom}{\mathrm{Hom}}

\newcommand{\Ker}{\textrm{Ker}}

\renewcommand{\P}{\mathbb{P}}

\newcommand{\Rep}{\mathrm{Rep}}

\newcommand{\Spec}{\mathrm{Spec}}
\newcommand{\Z}{\mathbb{Z}}

\begin{document}

\title{A Tannakian classification of torsors on the projective line}
\author{Johannes Ansch\"{u}tz}
\email{ja@math.uni-bonn.de}
\date{\today}

\begin{abstract}
In this small note we present a Tannakian proof of the theorem of Grothendieck-Harder on the classification of torsors under a reductive group on the projective line over a field.
\end{abstract}

\maketitle

\section{Introduction}

Let $k$ be a field, let $G/k$ be a reductive group and let $\P^1_k$ be the projective line over $k$.
In this small note we present a Tannakian proof of the classification of $G$-torsors on $\P^1_k$, thereby reproving known results of A.\ Grothendieck \cite{grothendieck_sur_la_classification_des_fibres} and G.\ Harder \cite[Satz 3.4.]{harder_halbeinfache_gruppenschemata_ueber_vollstaendigen_kurven}. 
To state our main theorem we denote by 
$$
\mathrm{Hom}^\otimes(\Rep_k(G),\Rep_k(\Gm))
$$
the set of isomorphism classes of exact tensor functors
$$
\omega\colon \Rep_k(G)\to \Rep_k(\Gm).
$$

\begin{theorem}[cf.\ \Cref{theorem: main theorem}, \Cref{corollary: zariski torsors}]
\label{theorem: main theorem introduction}
There exists a canonical bijection
$$
\mathrm{Hom}^\otimes(\Rep_k(G),\Rep_k(\Gm))\cong H^1_\et(\P^1_k,G).
$$
In particular, there exists a canonical bijection
$$
\Hom(\Gm,G)/{G(k)}\cong H^1_{\mathrm{Zar}}(\P^1_k,G).
$$
\end{theorem}

If $A\subseteq G$ denotes a maximal split torus, then
$$
\Hom(\Gm,G)/{G(k)}\cong X_\ast(A)_+
$$
is in bijection with the set of dominant cocharacters of $A\subseteq G$, which gives a very concrete description of the set $H^1_{\mathrm{Zar}}(\P^1_k,G)$. Using pure inner forms of $G$ over $k$ one can describe similarly the whole set $H^1_\et(\P^1_k,G)$ (cf.\ \Cref{lemma: fibers are zariski torsors}). 

Our proof of \Cref{theorem: main theorem introduction}, which originated in questions about torsors over the Fargues-Fontaine curve (cf.\ \cite{anschuetz_reductive_group_schemes_over_the_fargues_fontaine_curve}), is based on the Tannakian description of $G$-torsors (cf.\ \Cref{lemma: tannakian description of torsors}), the Tannakian theory of filtered fiber functors (cf.\ \cite{ziegler_graded_and_filtered_fiber_functors}), the canonicity of the Harder-Narasimhan filtration (cf.\ \Cref{lemma: harder-narasimhan filtration tensor functor}) and, most importantly, the good understanding of the category $\Bun_{\P^1_k}$ of vector bundles on $\P^1_k$ (cf.\ \Cref{theorem: classification of vector bundles}). In particular, we use crucially the fact that
$$
H^1_\et(\P^1_k,\mathcal{E})=0
$$ 
for $\mathcal{E}$ a semistable vector bundle on $\P^1_k$ of slope $\geq 0$.

In a last section we mention applications of \Cref{theorem: main theorem introduction} to the the computation of the Brauer group of $\P^1_k$ (avoiding Tsen's theorem) and to the Birkhoff-Grothendieck decomposition of $G(k((t)))$.

\subsection{Acknowledgment}
We want to thank Jochen Heinloth for his interest and for answering several questions.

\section{Vector bundles on $\P^1_k$}

Let $k$ be an arbitrary field. We recall, in a more canonical form, the classification of vector bundles on the projective line $\P^1_k$ due to A.\ Grothendieck (cf.\ \cite{grothendieck_sur_la_classification_des_fibres}).
Let 
$$
\Rep_k(\Gm)
$$ 
be the category of finite dimensional representations of the multiplicative group $\Gm$ over $k$. More concretely, the category $\Rep_k(\Gm)$ is equivalent to the Tannakian category of finite dimensional $\Z$-graded vector spaces over $k$.
 
Over $\P^1_k$ there is the canonical $\Gm$-torsor
$$
\eta\colon \mathbb{A}^2_k\setminus\{0\}\to \P^1_k,\ (x_0,x_1)\mapsto [x_0:x_1],
$$
also called the ``Hopf bundle''.
Given a representation $V\in \Rep_k(\Gm)$ the contracted product
$$
\mathcal{E}(V):=\mathbb{A}^2_k\setminus\{0\}\times^\Gm V\to \P^1_k
$$
defines a (geometric) vector bundle over $\P^1_k$. The well known classification of the category 
$$
\Bun_{\P^1_k}
$$ 
of vector bundles on $\P^1_k$ can now be phrased in the following way.

\begin{theorem}
\label{theorem: classification of vector bundles}
The functor
$$
\mathcal{E}(-)\colon \Rep_k(\Gm)\to \Bun_{\P^1_k}
$$
is an exact, faithful tensor functor inducing a bijection on isomorphism classes.  
\end{theorem}

However, the functor $\mathcal{E}(-)$ is not an equivalence. For example, by semi-simplicity of the category $\Rep_k(\Gm)$ every short exact sequence of $\Gm$-representations splits, but this is not true for short exact sequences of vector bundles on $\P^1_k$.

For $V\in \Rep_k(\Gm)$ the Harder-Narasimhan filtration of the vector bundle
$$
\mathcal{E}(V)
$$
has a very simple description. Namely, write 
$$
V=\bigoplus\limits_{i\in \Z} V_i
$$
with $\Gm$ acting on $V_i$ by the character\footnote{The sign is explained by the fact that the standard represention $z\mapsto z$ of $\Gm$ is sent by $\mathcal{E}(-)$ to $\mathcal{O}_{\P^1_k}(-1)$ and not to $\mathcal{O}_{\P^1_k}(1)$.} 
$$
\Gm\to \Gm,\ z\mapsto z^{-i}
$$
and set 
$$
\mathrm{fil}^i(V):=\bigoplus\limits_{j\geq i} V_j 
$$
for $i\in \Z$.
Then the Harder-Narasimhan filtration of $\mathcal{E}:=\mathcal{E}(V)$ is given by
$$
\ldots \subseteq \mathrm{HN}^{i+1}(\mathcal{E})\subseteq \mathrm{HN}^{i}(\mathcal{E})\subseteq\ldots \subseteq \mathcal{E}.
$$
where
$$
\mathrm{HN}^{i}(\mathcal{E}):=\mathcal{E}(\mathrm{fil}^{i}(V)).
$$

\begin{lemma}
\label{lemma: harder-narasimhan filtration tensor functor} 
Sending a vector bundle $\mathcal{E}$ to the filtered vector bundle $\mathcal{E}$ with the Harder-Narasimhan filtration $\mathrm{HN}^\bullet(\mathcal{E})$ defines a fully faithful tensor functor
$$
\mathrm{HN}\colon \Bun_{\P^1_k}\to \mathrm{FilBun}_{\P^1_k}
$$ 
into the exact tensor category of filtered vector bundles (with filtration by locally direct summands) (cf.\ \cite[Chapter 4]{ziegler_graded_and_filtered_fiber_functors} for a definition of $\mathrm{FilBun}_{\P^1_k}$).
\end{lemma}
\begin{proof}
This is clear from the description of the Harder-Narasimhan filtration.
\end{proof}

We remark that the functor $\mathrm{HN}$ is \textit{not} exact as one sees for example by looking at the Euler sequence
$$
0\to \mathcal{O}_{\P^1_k}(-1)\to \mathcal{O}_{\P^1_k}\oplus \mathcal{O}_{\P^1_k}\to \mathcal{O}_{\P^1_k}(1)\to 0
$$
on $\P^1_k$.

Sending a filtered vector bundle $(\mathcal{E}, F^\bullet)$ to the associated graded vector bundle
$$
\mathrm{gr}(\mathcal{E}):=\bigoplus_{i\in \Z} F^i\mathcal{E}/F^{i+1}\mathcal{E}
$$
defines an exact tensor functor
$$
\mathrm{gr}\colon \mathrm{FilBun}_{\P^1_k}\to \mathrm{GrBun}_{\P^1_k}
$$
(cf.\ \cite[Chapter 4]{ziegler_graded_and_filtered_fiber_functors}).

The following lemma is immediate from \Cref{theorem: classification of vector bundles}, \Cref{lemma: harder-narasimhan filtration tensor functor} and the fact that
$$
H^0(\P^1_k,\mathcal{O}_{\P^1_k})\cong k.
$$

\begin{lemma}
\label{lemma: representations of gm as graded vector bundles}
The composite functor
$$
\Rep_k(\Gm)\xrightarrow{\mathcal{E}(-)} \Bun_{\P^1_k}\xrightarrow{\mathrm{HN}} \mathrm{FilBun}_{\P^1_k}\xrightarrow{\mathrm{gr}} \mathrm{GrBun}_{\P^1_k}
$$  
is an equivalence of exact categories from $\Rep_k(\Gm)$ onto its essential image which consists of graded vector bundles 
$$
\mathcal{E}=\bigoplus\limits_{i\in \Z} \mathcal{E}^i
$$
such that each $\mathcal{E}^i$ is semistable of slope $i$.
\end{lemma}

\section{Torsors over $\P^1_k$}

Let $G/k$ be an arbitrary reductive group. In this section we want to classify $G$-torsors on $\P^1_k$ for the \'etale topology. For this we keep the notation from the last section.
In particular, there is the functor
$$
\mathcal{E}(-)\colon \Rep_k(\Gm)\to\Bun_{\P^1_k}
$$
from \Cref{theorem: classification of vector bundles}

In order to apply the formulations from the previous section we need a more bundle theoretic interpretation of $G$-torsors (for the \'etale topology). This is achieved by the Tannakian formalism (cf.\ \cite{deligne_categories_tannakiennes})

\begin{lemma}
\label{lemma: tannakian description of torsors}
Let $S$ be a scheme over $k$. Sending a $G$-torsor $\mathcal{P}$ over $S$ to the exact tensor functor
$$
\omega\colon \Rep_k(G)\to \Bun_{S},\ V\mapsto \mathcal{P}\times^{G}(V\otimes_k \mathcal{O}_S)
$$
defines an equivalence from the groupoid of $G$-torsors to the groupoid of exact tensor functors from $\Rep_k(G)$ to $\Bun_S$.
The inverse equivalence sends an exact tensor functor $\omega\colon \Rep_k(G)\to \Bun_{S}$ the $G$-torsor $\mathrm{Isom}^\otimes(\omega_{\mathrm{can}},\omega)$ of isomorphisms of $\omega$ to the canonical fiber functor $\omega_{\mathrm{can}}\colon \Rep_k(G)\to \Bun_S,\ V\mapsto V\otimes_k \mathcal{O}_S$.
\end{lemma}

In fact, for a general affine group scheme over $k$ one has to use the fpqc-topology in \Cref{lemma: tannakian description of torsors}. However, as $G$ is assumed to be reductive, thus in particular smooth, a theorem of Grothendieck (cf.\ \cite[Th\'eor\`eme 11.7]{grothendieck_le_group_de_brauer_III}) allows to reduce to the \'etale topology.

Composing an exact tensor functor
$$
\omega\colon \Rep_k(G)\to \Bun_{\P^1_k} 
$$ 
with the Harder-Narasimhan functor
$$
\mathrm{HN}\colon \Bun_{\P^1_k}\to \mathrm{FilBun}_{\P^1_k} 
$$
defines a, a priori not necessarily exact, tensor functor
$$
\mathrm{HN}\circ \omega\colon \Rep_k(G)\to \mathrm{FilBun}_{\P^1_k}.
$$

But using Haboush's theorem reductivity of $G$ actually implies that the composition $\mathrm{HN}\circ \omega$ is still exact.

\begin{lemma}
\label{lemma: composition with harder-narasimhan filtration still exact}
Let 
$$
\omega\colon \Rep_k(G)\to \Bun_{\P^1_k} 
$$ 
be an exact tensor functor. Then the composition 
$$
\mathrm{HN}\circ \omega\colon \Rep_k(G)\to \mathrm{FilBun}_{\P^1_k}
$$
is still exact.
\end{lemma}
\begin{proof}
The crucial observation is that the functors 
$$
\omega,\ \mathrm{gr}\circ \mathrm{HN}
$$
are compatible with duals, and exterior resp.\ symmetric products. This is clear for $\omega$ as $\omega$ is assumed to be exact and follows from \Cref{lemma: representations of gm as graded vector bundles} for the functor $\mathrm{HN}\circ \mathrm{gr}$. In fact, for a representation $V\in \Rep_k(\Gm)$ with associated vector bundle
$$
\mathcal{E}:=\mathcal{E}(V)
$$
we can conclude
$$
\Lambda^r(\mathcal{E})\cong \mathcal{E}(\Lambda^r(V)) \textrm{ resp.\ } \mathrm{Sym}^r(\mathcal{E})\cong \mathcal{E}(\mathrm{Sym}^r(V))
$$
by exactness of the functor $\mathcal{E}(-)$. But by \Cref{lemma: representations of gm as graded vector bundles} 
$$
\mathrm{gr}\circ \mathrm{HN}\circ \mathcal{E}(-)
$$
is an exact tensor equivalence of $\Rep_k(\Gm)$ with a subcategory of $\mathrm{GrBun}_{\P^1_k}$, which implies the stated compatibility with exterior and symmetric powers.
Using this the proof can proceed similarly to \cite[Theorem 5.3.1]{dat_orlik_rapoport_period_domains}.
We note that for a representation $V$ of $G$ there is a canonical isomorphism
$$
\mathrm{Sym}^r(V^\vee)\cong \mathrm{TS}_r(V)^\vee
$$
from the $r$-th symmetric power $\mathrm{Sym}^r(V^\vee)$ of the dual of $V$ to the dual of the module 
$$\mathrm{TS}_r(V)=(V^{\otimes r})^{S_r}\subseteq V^{\otimes r}
$$ 
of symmetric tensors.
In particular, $G$-invariant homogenous polynomials on $V$ define $G$-invariant linear forms on $\mathrm{TS}_r(V)^\vee$.

Let now $0\to V\xrightarrow{f} V^\prime\xrightarrow{g} V^{\prime\prime}\to 0$ be an exact sequence in $\Rep_k(G)$. We have to check that the sequence
$$
0\to \tilde{\omega}(V)\xrightarrow{\tilde{\omega}(f)} \tilde{\omega}(V^\prime)\xrightarrow{\tilde{\omega}(g)} \tilde{\omega}(V^{\prime\prime})\to 0
$$
with 
$$
\tilde{\omega}:=\mathrm{gr}\circ \mathrm{HN}\circ \omega
$$
is still exact.
We claim that $\tilde{\omega}(f)$ is injective. This can be checked after taking the exterior power $\Lambda^{\dim V}$ of $f$ because $\tilde{\omega}$ commutes with exterior powers. In particular, to prove injectivity we can reduce the claim for general $f$ to the case $\dim V=1$. Tensoring with the dual of $V$ reduces further to the case the $V$ is moreover trivial. By Haboush's theorem (cf.\ \cite{haboush_reductive_groups_are_geometrically_reductive}) there exists an $r>0$ and a $G$-invariant homogenous polynomial $f\in \mathrm{Sym}^r(V^\vee)$ such that $f_{|V}\neq 0$. Using the above isomorphism $\mathrm{Sym}^r(V^\vee)\cong \mathrm{TS}_r(V)^\vee$ this shows that there exists an $r> 0$ such that the morphism
$$
V\cong \mathrm{TS}_r(V)\xrightarrow{\mathrm{TS}_r(f)}\mathrm{TS}_r(V^\prime)
$$
splits. This implies that $\tilde{\omega}(\mathrm{TS}_r(f))$ splits and thus that $\tilde{\omega}(f)$ is in particular injective because $\tilde{\omega}$ commutes with the symmetric tensors $TS_r$ as it commutes with symmetric powers and duals.

Dualizing yields that $\tilde{\omega}(g)$ is surjective at the generic point of $\P^1_k$. However, the sequence
$$
0\to \tilde{\omega}(V)\xrightarrow{\tilde{\omega}(f)} \tilde{\omega}(V^\prime)\xrightarrow{\tilde{\omega}(g)} \tilde{\omega}(V^{\prime\prime})\to 0
$$
lies in the essential image of the functor $\Rep_k(\Gm)\to \mathrm{GrBun}_{\P^1_k}$ from \Cref{lemma: representations of gm as graded vector bundles}. In particular, we see that the cokernel of $\tilde{\omega}(g)$ cannot have torsion, i.e., that it is zero.
Finally, exactness in the middle of the sequence follows because
$$
\mathrm{rk}(\tilde{\omega}(V^\prime))=\mathrm{rk}(V^\prime)=\mathrm{rk}(V)+\mathrm{rk}(V^{\prime\prime})=\mathrm{rk}(\tilde{\omega}(V))+\mathrm{rk}(\tilde{\omega}(V^{\prime\prime})).
$$
This finishes the proof.
\end{proof}

We briefly recall some results about filtered fiber functors on $\Rep_kG$ (cf.\ \cite{ziegler_graded_and_filtered_fiber_functors} and \cite{cornut_filtrations_and_buildings}). By definition a filtered fiber functor for $\Rep_kG$ over a $k$-scheme $S$ is an exact tensor functor
$$
\omega\colon \Rep_k G\to \mathrm{FilBun}_S
$$
into the exact tensor category of filtered vector bundles (with filtration by locally direct summands) on $S$. Associated to each filtered fiber functor $\omega$ is an exact tensor functor 
$$
\mathrm{gr}\circ \omega\colon \Rep_k G\to \mathrm{GrBun}_S,
$$
i.e., a graded fiber functor, by mapping a filtered vector bundle to its associated graded.
A splitting $\gamma$ of a filtered fiber functor $\omega$ is a graded fiber functor
$$
\gamma\colon \Rep_kG \to \mathrm{GrBun}_S
$$
such that
$$
\omega=\mathrm{fil}\circ \gamma
$$
where the exact tensor functor
$$
\mathrm{fil}\colon \mathrm{GrBun}_S\to \mathrm{FilBun}_S
$$
sends a graded vector bundle 
$$
\mathcal{E}=\bigoplus\limits_{i\in \Z}\mathcal{E}^i
$$
to the filtered vector bundle $(\mathcal{E}, \mathrm{fil}^\bullet \mathcal{E})$ with filtration
$$
\mathrm{fil}^i\mathcal{E}=\bigoplus\limits_{j\geq i}\mathcal{E}^j.
$$
For a scheme $f\colon S^\prime\to S$ over $S$ let $\omega_{S^\prime}$ be the base change of the filtered fiber functor $\omega$ to $S^\prime$, i.e., $\omega_{S^\prime}$ is defined as the composition
$$
\Rep_k G\xrightarrow{\omega} \mathrm{FilBun}_{S}\xrightarrow{f^\ast} \mathrm{FilBun}_{S^\prime},
$$
which is again a filtered fiber functor.
For a filtered fiber functor $\omega$ the presheaf
$$
\mathrm{Spl}(\omega)(S^\prime):=\{ \textrm{ set of splittings of }  \omega_{S^\prime}\}
$$
on the category of $S$-schemes is represented by an fpqc-torsor for the affine and faithfully flat group scheme 
$$
U(\omega):=\Ker(\Aut^\otimes(\omega)\to \Aut^\otimes(\mathrm{gr}\circ \omega))
$$
over $S$ (cf.\ \cite[Lemma 4.20]{ziegler_graded_and_filtered_fiber_functors}).
In particular, every filtered fiber functor 
$$
\omega\colon \Rep_k G\to \mathrm{FilBun}_S
$$ 
admits a splitting fpqc-locally on $S$.
The group scheme $U(\omega)$ can be described more explicitely (cf.\ \cite[Theorem 4.40]{ziegler_graded_and_filtered_fiber_functors}). Namely there exists a decreasing filtration by normal subgroups
$$
U(\omega)=U_1(\omega)\supseteq\ldots \supseteq U_i(\omega)\supseteq \ldots
$$
for $i\geq 1$, which has the property that for $i\geq 1$ the quotient
$$
\mathrm{gr}^iU(\omega):= U_i(\omega)/U_{i+1}(\omega)
$$
is abelian and isomorphic to
$$
\mathrm{gr}^iU(\omega)\cong \mathrm{Lie}(\mathrm{gr}^iU(\omega))\cong \mathrm{gr}^i\omega(\mathrm{Lie}(G)),\ i\geq 1.
$$
We can now give a proof of our main theorem about the classification of $G$-torsors on $\mathbb{P}^1_k$. We denote for a scheme $S$ over $k$ by
$$
\underline{\Hom}^\otimes(\Rep_k(G), \Bun_S)
$$
the groupoid of exact tensor functors $\omega\colon \Rep_k(G)\to \Bun_S$ and by
$$
\Hom^\otimes(\Rep_k(G), \Bun_S)
$$
its set of isomorphism classes. Similarly, we use the notations
$$
\underline{\Hom}^\otimes(\Rep_k(G), \Rep_k(\Gm))
$$
resp.\
$$
\Hom^\otimes(\Rep_k(G), \Rep_k(\Gm))
$$
for the groupoid resp.\ the isomorphism classes of exact tensor functors
$$
\omega\colon \Rep_k(G)\to \Rep_k(\Gm).
$$

\begin{theorem}
\label{theorem: main theorem}
Let $G$ be a reductive group over $k$.
Then the composition with $\mathcal{E}(-)$ defines faithful functor 
$$
\Phi\colon \underline{\Hom}^\otimes(\Rep_k(G),\Rep_k(\Gm))\to \underline{\Hom}^\otimes(\Rep_k(G),\Bun_{\P^1_k})
$$
which induces a bijection
$$
\Hom^\otimes(\Rep_k(G),\Rep_k(\Gm))\cong H^1_\et(\P^1_k,G).
$$
on isomorphism classes.
\end{theorem}
\begin{proof}
By \Cref{lemma: representations of gm as graded vector bundles} the composition
$$
\Rep_k(\Gm)\xrightarrow{\mathcal{E}(-)} \Bun_{\P^1_k}\xrightarrow{\mathrm{HN}} \mathrm{FilBun}_{\P^1_k}\xrightarrow{\mathrm{gr}} \mathrm{GrBun}_{\P^1_k}
$$  
is an equivalence onto its essential image. In particular, the functor
$$
\Phi\colon \underline{\Hom}^\otimes(\Rep_k(G),\Rep_k(\Gm))\to \underline{\Hom}^\otimes(\Rep_k(G),\Bun_{\P^1_k})
$$
is faithful and induces an injection on isomorphism classes.
Thus we have to prove that every exact tensor functor
$$
\omega\colon \Rep_k(G)\to \Bun_{\P^1_k}
$$
factors as
$$
\omega\cong \mathcal{E}(-)\circ \omega^\prime
$$
for some exact tensor functor
$$
\omega^\prime\colon \Rep_k(G)\to \Rep_k(\Gm).
$$
Let $\tilde{\omega}:=\mathrm{HN}\circ \omega$ be the functor
$$
\tilde{\omega}\colon \Rep_k(G)\xrightarrow{\omega}\Bun_{\P^1_k}\xrightarrow{\mathrm{HN}} \mathrm{FilBun}_{\P^1_k}.
$$
By \Cref{theorem: main theorem} the functor $\tilde{\omega}$ is still exact, i.e., a filtered fiber functor in the terminology of \cite{ziegler_graded_and_filtered_fiber_functors}, and we can use the results recalled above. 
We get a $U(\tilde{\omega})$-torsor 
$$
\mathrm{Spl}(\tilde{\omega})
$$
of splittings of $\tilde{\omega}$. 
But for the filtration
$$
U(\tilde{\omega})\supseteq U_{2}(\tilde{\omega})\supseteq \ldots
$$
the graded quotients
$$
\mathrm{gr}^iU(\tilde{\omega})\cong \mathrm{gr}^i\tilde{\omega}(\mathrm{Lie}(G))
$$
are semistable vector bundles of slope $i\geq 1$.
Hence,
$$
H^1_{\et}(\P^1_k,\mathrm{gr}^iU(\tilde{\omega}))=0
$$
because 
$$
\mathrm{gr}^iU(\tilde{\omega})\cong  \mathcal{O}_{\P^1_k}(i)^{\oplus n}
$$
by \Cref{theorem: classification of vector bundles}.
We can conclude that
$$
H^1_\et(\P^1_k, U(\tilde{\omega}))=1,
$$
hence the $U(\tilde{\omega})$-torsor
$$
\mathrm{Spl}(\tilde{\omega})
$$
is in fact trivial, i.e., there exists a splitting
$$
\gamma\colon \Rep_k G\to \mathrm{GrBun}_{\P^1_k}
$$
of $\tilde{\omega}$ already over $\P^1_k$. As 
$$
\gamma\cong \mathrm{gr}\circ \tilde{\omega}
$$ 
the functor $\gamma$ takes its image in the full subcategory
$$
\{\ \mathcal{E}=\bigoplus\limits_{i\in \Z} \mathcal{E}^i\in \mathrm{GrBun}_{\P^1}\ |\ \mathcal{E}^i \textrm{ semistable of slope }i\},
$$
which by \Cref{lemma: representations of gm as graded vector bundles} is equivalent to the category $\Rep_k \Gm$ of representations of $\Gm$. Thus there exists an exact tensor functor
$$
\omega^\prime\colon \Rep_k G\to \Rep_k \Gm
$$
such that 
$$
\omega\cong \mathcal{E}(-)\circ \omega^\prime,
$$
by simply setting
$$
\omega^\prime:=\mathcal{E}_{\mathrm{gr}}(-)^{-1}\circ \mathrm{gr}\circ \tilde{\omega}
$$
where 
$$
\mathcal{E}_{\mathrm{gr}}(-)\colon \Rep_k \Gm\to \{\ \mathcal{E}=\bigoplus\limits_{i\in \Z} \mathcal{E}^i\in \mathrm{GrBun}_{\P^1}\ |\ \mathcal{E}^i \textrm{ semistable of slope }i\},
$$
is the the equivalence of \Cref{lemma: representations of gm as graded vector bundles}.
\end{proof}

Let 
$$
\omega_{\mathrm{can}}^\Gm\colon \Rep_k(\Gm)\to \mathrm{Vec}_k,\ V\mapsto V
$$
be the canonical fiber functor of $\Rep_k(\Gm)$ over $k$.
Composing with $\omega_{\mathrm{can}}^\Gm$ defines a morphism 
$$
\Phi\colon \underline{\Hom}^\otimes(\Rep_k(G),\Rep_k(\Gm))\to \underline{\Hom}^\otimes(\Rep_k(G),\mathrm{Vec}_k)
$$
of groupoids, where the right hand side denotes the groupoid of exact tensor functors
$$
\Rep_k(G)\to \mathrm{Vec}_k,
$$
which by \Cref{lemma: tannakian description of torsors} identifies with the groupoid of $G$-torsors on $\Spec(k)$.
Geometrically, the morphism $\Phi$ can be identified on isomorphisms classes with the map
$$
i^\ast\colon H^1_{\et}(\P^1_k, G)\to H^1_\et(\Spec(k),G) 
$$ 
restricting a $G$-torsor over $\P^1_k$ to a $G$-torsor over $\Spec(k)$ along a $k$-rational point $x\in \P^1_k(k)$.

In the following lemma we analyze the fibers of this functor $\Phi$.

\begin{lemma}
\label{lemma: fibers are zariski torsors}
Let $\omega\colon \Rep_k(G)\to \mathrm{Vec}_k$ be an exact tensor functor and let 
$$
H:=\Aut^\otimes(\omega)
$$ 
be the pure inner form of $G$ defined by $\omega$.
Then the fiber 
$$
\Phi^{-1}(\omega)\subseteq \underline{\Hom}^\otimes(\Rep_kG,\Rep_k \Gm)
$$
is equivalent to the quotient groupoid
$$
[\Hom(\Gm,H)/{H(k)}]
$$
of cocharacters of $H$. Moreover, passing to isomorphism classes yields a bijection
$$
\Hom(\Gm,H)/{H(k)}\cong H^1_{\mathrm{Zar}}(\P^1_k,H).
$$
\end{lemma}
\begin{proof}
The first statement follows from the Tannakian formalism (cf.\ \cite{deligne_categories_tannakiennes}). Namely, $\omega$ defines an equivalence
$$
\Rep_k(G)\cong \Rep_k(H),\ V\mapsto \omega(V)
$$
and the groupoid of exact tensor functors
$$
\Rep_k H\to \Rep_k \Gm
$$
which commute (with a given isomorphism) with the canonical fiber functors on $\Rep_k H$ resp.\ $\Rep_k \Gm$ is equivalent to the quotient groupoid
$$
[\Hom(\Gm,H)/{H(k)}].
$$
with $H(k)$ acting by conjugation.
Clearly, for every cocharacter 
$$
\chi\colon \Gm\to H
$$
the push forward
$$
\eta\times^{\Gm}H
$$ 
is an $H$-torsor, which is locally trivial in the Zariski topology, because this is true for the Hopf bundle
$$
\eta\colon \mathbb{A}^2_k\setminus\{0\}\to \P^1_k.
$$ 
Let conversely $\mathcal{P}$ be an $H$-torsor over $\P^1_k$ which is trivial for the Zariski topology and let
$$
\omega_{\mathcal{P}}\colon \Rep_k H\to \Bun_{\P^1_k}, V\mapsto \mathcal{P}\times^{H}(V\otimes_k \mathcal{O}_{\P^1_k})
$$
be the induced fiber functor (cf.\ \Cref{lemma: tannakian description of torsors}).
Let $x\in \P^1_k(k)$ be a point a $k$-rational point and let $U\subseteq \P^1_k$ be open subset containing $x\in U$ such that 
$$
\mathcal{P}_{|U}
$$
is trivial. Then the exact tensor functor
$$
\Rep_k H\xrightarrow{\omega_{\mathcal{P}}}\Bun_{\P^1_k}\xrightarrow{\mathrm{res}}\Bun_U
$$
is isomorphic to the trivial fiber functor. This holds then also true after restricting to $x\in U$.
Let
$$
\varphi\colon \Rep_k H\to \Rep_k \Gm
$$
be an exact tensor functor such that 
$$
\mathcal{E}(-)\circ \varphi \cong \omega_{\mathcal{P}}.
$$
We can conclude that $\varphi$ preserves the canonical fiber functors on $\Rep_k H$ resp.\ $\Rep_k \Gm$ because the composition
$$
\Rep_k \Gm\xrightarrow{\mathcal{E}(-)} \Bun_{\P^1_k}\xrightarrow{\mathrm{res}}\Bun_{x}\cong \mathrm{Vec}_k
$$
is the canonical fiber functor.
In particular, there exists a cocharacter
$$
\chi\colon \Gm\to H
$$
such that $\mathcal{P}$ is obtained via pushout along $\chi$ of the Hopf bundle
$$
\eta\colon \mathbb{A}^2_k\setminus\{0\}\to \P^1_k.
$$ 
\end{proof}

Note that we have actually shown that a $G$-torsor $\mathcal{P}$ is already locally trivial for the Zariski topology if there exists some open $U\subseteq \P^1_k$ containing a $k$-rational point, such that $\mathcal{P}_{|U}$ is trivial.   
The classification results of Grothendieck and Harder on torsors on $\P^1_k$ (cf.\ \cite{grothendieck_sur_la_classification_des_fibres} resp.\ \cite{harder_halbeinfache_gruppenschemata_ueber_vollstaendigen_kurven}) are most concretely stated in the collowing form.

\begin{corollary}
\label{corollary: zariski torsors}
Let $k$ be a field and let $G/k$ be a reductive group with maximal split subtorus $A\subseteq G$. Then there exist canonical bijections
$$
X_\ast(A)_+\cong \mathrm{Hom}(\Gm,G)/G(k)\cong H^1_{\mathrm{Zar}}(\P^1_k,G),
$$ 
where $X_\ast(A)_+$ denotes the set of dominant cocharacters of $A\subseteq G$.
\end{corollary}
\begin{proof}
By \Cref{lemma: fibers are zariski torsors} it suffices to show
$$
X_\ast(A)_+\cong \mathrm{Hom}(\Gm,G)/G(k).
$$
But this follows from the fact that all maximal split tori in $G$ are conjugated over $k$  and that the set of dominant cocharacters form a system of representatives for the action of the normalizer $N_G(A)$ of $A$ in $G$ on the group $X_\ast(A)$ of cocharacters for $A$.
\end{proof}

A description of $H^1_\et(\P^1_k,G)$, similar to the one of us, can be found in \cite{gille_torseurs_sur_la_droite_affine}.

\section{Applications}

In this section we present some applications of the classification of torsors (following (cf.\ \cite{fargues_g_torseurs_en_theorie_de_hodge_p_adique}, which discusses analogous applications to the Fargues-Fontaine curve).

The first application is the computation of the Brauer group of $\P^1_k$.
For this we recall the theorem of Steinberg (cf.\ \cite[Chapter 3.2.3]{serre_galois_cohomology}). If $k$ is a field of cohomological dimension $\mathrm{cd}(k)\leq 1$, then Steinberg's theorem states that
$$
H^1_{\et}(\Spec(k),G)=1
$$ 
for every smooth connected affine algebraic group $G/k$.
In particular, the Brauer group
$$
\mathrm{Br}(k)=0
$$ of such fields vanishes.
For example, separably closed or finite fields are of cohomological dimension $\leq 1$.

\begin{theorem}
\label{theorem: brauer group of p1}
If $k$ is of cohomological dimension $\mathrm{cd}(k)\leq 1$, then the Brauer group
$$
\mathrm{Br}(\P^1_k)\cong H^2_{\et}(\P^1_k,\Gm)=0
$$
vanishes.  
\end{theorem}
\begin{proof}
By \cite[Corollaire 2.2.]{grothendieck_le_group_de_brauer_II} there is an isomorphism
$$
\mathrm{Br}(\P^1_k)\cong H^2_{\et}(\P^1_k,\Gm)
$$
of the Brauer group $\mathrm{Br}(\P^1_k)$ parametrizing equivalence classes of Azumaya algebras over $\mathcal{O}_{\P^1_k}$ with the cohomological Brauer group $H^2_\et(\P^1_k,\Gm)$.
It suffices to show that for every $n\geq 0$ the canonical map
$$
H^1_\et(\P^1_k,\mathrm{PGL}_n)\to H^2_{\et}(\P^1_k,\Gm)
$$
arising as a boundary map of the short exact sequence
$$
1\to \Gm\to \mathrm{GL}_n\to \mathrm{PGL}_n\to 1
$$
is trivial. Because $k$ is of cohomological dimension $\leq 1$, there exists using Steinberg's theorem in the case $G=\mathrm{GL}_n$ or $G=\mathrm{PGL}_n$ and \Cref{theorem: main theorem} together with \Cref{lemma: fibers are zariski torsors} a commutative diagram
$$
\xymatrix{
H^1_{\et}(\P^1_k, \mathrm{GL}_n)\ar[r]\ar[d]^\cong & H^1_\et(\P^1_k,\mathrm{PGL}_n)\ar[d]^\cong \\
\Hom(\Gm,\mathrm{GL}_n)/{\mathrm{GL}_n(k)}\ar[r] & \Hom(\Gm,\mathrm{PGL}_n)/{\mathrm{PGL}_n(k)}.
}
$$
It suffices to show that the top horizontal arrow, or equivalently the lower horizontal arrow, is surjective. But every cocharacter 
$$
\chi\colon \Gm\to \mathrm{PGL}_n  
$$
can be lifted to $\mathrm{GL}_n$ because for the standard torus $T\cong \mathbb{G}_m^n\subseteq \mathrm{GL}_n$ there is a split exact sequence
$$
0\to X_\ast(\Gm)\to X_\ast(T)\to X_\ast(T/\Gm)\to 0
$$
on cocharacter groups where $T/\Gm$ is a maximal torus of $\mathrm{PGL}_n$.
\end{proof}

For a general field $k$, i.e., $k$ not necessarily of cohomological dimension $\leq 1$, the Brauer group of $\P^1_k$ is given by 
$$
\mathrm{Br}(\Spec(k))\cong \mathrm{Br}(\P^1_k)
$$
as can be calculated from \Cref{theorem: brauer group of p1} using the spectral sequence
$$
E^{pq}_2=H^p(\mathrm{Gal}(\bar{k}/k),H^q_{\et}(\P^1_{\bar{k}},\Gm))\Rightarrow H^{p+q}_{\et}(\P^1_k,\Gm)
$$
where $\bar{k}$ denotes a separable closure of $k$.

The next application we give is to the uniformization of $G$-torsors.

\begin{theorem}
\label{theorem: uniformization}
Let $k$ be a field and let $G$ be reductive group over $k$.
If $x\in \P^1_k(k)$ is $k$-rational point, then every $G$-torsor 
$$
\mathcal{P}\in H^1_{\mathrm{Zar}}(\P^1_k,G)
$$ 
which is locally trivial for the Zariski topology becomes trivial on $\P^1_k\setminus\{x\}$.
\end{theorem}
\begin{proof}
By \Cref{corollary: zariski torsors} we know that every such $G$-torsor $\mathcal{P}$ is isomorphic to the pushout
$$
\mathcal{P}\cong \eta\times^\Gm G
$$ along a cocharacter
$$
\chi\colon \Gm\to G
$$ 
of the canonical $\Gm$-torsor
$$
\eta\colon \mathbb{A}^2_k\setminus\{0\}\to \P^1_k
$$
corresponding to the line bundle $\mathcal{O}_{\P^1_k}(-1)$ on $\P^1_k$. But 
$$
\mathcal{O}_{\P^1_k}(-1)_{|\P^1_k\setminus\{x\}}
$$ 
is trivial because $\P^1_k\setminus\{x\}\cong \mathbb{A}^1_k$. This shows the claim.  
\end{proof}

Finally, we reprove the Birkhoff-Grothendieck decomposition of $G(k((t))$ for a reductive group $G$ over $k$ (cf.\ \cite[Lemma 4]{faltings_algebraic_loop_groups_and_moduli_spaces_of_bundles}).

\begin{theorem}
\label{theorem: decomposition}
Let $A\subseteq G$ be a maximal split torus in $G$. Then there exists a canonical bijection
$$
X_\ast(A)_{+}\cong {G(k[t^{-1}])}\backslash G(k((t)))/{G(k[[t]])},
$$  
where $X_\ast(A)_+$ denotes the set of dominant cocharacters of $A\subseteq G$.
\end{theorem}
\begin{proof}
Let $x\in \P^1_k(k)$ be a $k$-rational point.
By Beauville-Laszlo \cite{beauville_laszlo_un_lemme_de_descente} and \Cref{lemma: tannakian description of torsors} there is an injective map
$$
\gamma\colon {G(k[t^{-1}])}\backslash G(k((t)))/{G(k[[t]])}\to H^1_{\et}(\P^1_k,G)
$$
by glueing the trivial $G$-torsor on $\P^1_k\setminus\{x\}$ with the trivial $G$-torsor on the formal completion
$$
\Spec(\widehat{\mathcal{O}}_{\P^1_k,x})
$$ 
along an isomorphism on $\Spec(\mathrm{Frac}(\widehat{\mathcal{O}}_{\P^1_k,x}))$. Note that $\widehat{\mathcal{O}}_{\P^1_k,x}\cong k[[t]]$.
From the remark following \Cref{lemma: fibers are zariski torsors} we can conclude that the $G$-torsors obtained in this way are actually locally trivial for the Zariski topology. By \Cref{theorem: uniformization} we can conversely see that the image of $\gamma$ contains the set $H^1_{\mathrm{Zar}}(\P^1_k,G)$. Using \Cref{corollary: zariski torsors} we can conclude that
$$
{G(k[t^{-1}])}\backslash G(k((t)))/{G(k[[t]])}\cong H^1_{\mathrm{Zar}}(\P^1_k,G)\cong X_{\ast}(A)_+.
$$ 
\end{proof}

\end{document}